\documentclass{amsart}
\usepackage[margin=2.63cm]{geometry}
\pdfoutput=1
\usepackage[utf8]{inputenc} 

\usepackage[T1]{fontenc}    

\usepackage{url}

\usepackage{amsmath}
\usepackage{amsfonts}
\usepackage{amssymb}
\usepackage{amsthm} %
\usepackage{mathrsfs} %
\usepackage{enumerate} 
\usepackage{stmaryrd}
\usepackage{dirtytalk}

\usepackage{tikz}
\usepackage{pgfplots}
\usetikzlibrary{calc}
\usetikzlibrary{shapes}

\usetikzlibrary{intersections}
\usetikzlibrary{patterns}
\usepackage{comment}
\usepackage[all]{xy}

\usepackage{graphicx}
\usepackage{caption}
\definecolor{byzantine}{rgb}{0.74, 0.2, 0.64}
\definecolor{magenta}{rgb}{1.0, 0.0, 1.0}
\definecolor{islamicgreen}{rgb}{0.0, 0.56, 0.0}
\definecolor{ferrarired}{rgb}{1.0, 0.11, 0.0}

\definecolor{crimson}{rgb}{0.86, 0.08, 0.24}
\definecolor{applegreen}{rgb}{0.55, 0.71, 0.0}
\definecolor{ao}{rgb}{0.0, 0.5, 0.0}

\usepackage[hyperfootnotes,colorlinks=true,citecolor=cyan,backref=page]{hyperref}

\theoremstyle{plain} 
\newtheorem{thm}{Theorem}[section]

\newtheorem{prop}[thm]{Proposition}

\theoremstyle{definition}

\newtheorem*{question}{Question}
\newtheorem{ex}[thm]{Example}

\newcommand{\cone}{\operatorname{cone}}
\newcommand{\shi}{{\operatorname{Shi}}}

\author[N. Chapelier-Laget]{Nathan~Chapelier-Laget}
\address[Nathan Chapelier-Laget]{University of Sydney}
\email{nathan.chapelier@gmail.com}
\urladdr{https://www.nathanchapelier.fr/home}

\title{Low elements in dominant Shi regions}

\begin{document}
\maketitle

\begin{abstract}
This note is a complement of \cite[Section 6]{chaphohl2022shi}. We explain in terms of ad-nilpotent ideals of a Borel subalgebra why the minimal elements of dominant Shi regions are low.  We also give a survey of the bijections involved in the study of dominant Shi regions in affine Weyl groups. 
\end{abstract}

\section{Introduction}

 Let $(W,S)$ be Coxeter system with length function $\ell:W\to \mathbb N$.  Let $\Phi$ be a root system of $(W,S)$, with simple system $\Delta=\{\alpha_s\mid s\in S\}$ and positive root system $\Phi^+=\cone(\Delta)\cap \Phi$, where $\cone(X)$ is the set of nonnegative  linear combinations of vectors in $X$. The {\em inversion set of $w\in W$} is the set 
$
N(w)=\Phi^+\cap w(\Phi^-),
$
where $\Phi^-=-\Phi^+$; moreover $\ell(w)=|N(w)|$.

 In \cite{BrHo93} Brink and Howlett introduced the dominance order: this is the partial order $\preceq$ on $\Phi^+$ defined by
$$
\alpha \preceq \beta \Longleftrightarrow \forall w \in W, \beta \in N(w) \Longrightarrow \alpha \in N(w).
$$
For $\beta \in \Phi^+$, the \emph{dominance set} of $\beta$ is Dom($\beta$) := $\{ \alpha \in \Phi^+~|~\alpha \prec \beta \}$. The $\infty$-\emph{depth} on $\Phi^+$ is defined by dp$_{\infty}(\beta)$ = $|$Dom($\beta$)$|$. We say that $\beta$ is a \emph{small root} if dp$_{\infty}(\beta)$ = 0. The set of small roots is denoted $\Sigma$. 

 For $X \subset V$ we denote  $\text{cone}_{\Phi}(X) = \text{cone}(X) \cap \Phi$.  An important family of elements connected to $\Sigma$ is the set of low elements. Low elements are defined as follows: $w\in W$ is a {\em low element} if we have
$
 N(w)=\cone_{\Phi}(\Sigma\cap N(w)).
 $
 Denote by $L(W,S)$, or just $L$, the set of low elements of $(W,S)$.  A convenient characterization of being low is the following: $w\in W$ is low if there exists $X\subseteq \Sigma$ such that $N(w)=\cone_\Phi(X)$. 
 
 \subsection{Tits cone and Shi arrangement}

Let $X_S$ be the $ \mathbb{R}$-vector space with basis
$\{\alpha_s~|~s \in S\}$.  Let $B$ be the symmetric bilinear form on $X_S$ defined by 
$$
 B(e_s,e_t)=  \left\{
                          	\begin{array}{ll}
 						  -\text{cos}(\frac{\pi}{m_{st}})  & \text{if}~~ m_{st} < \infty \\
  				      	~~~~-1      & \text{if}~~m_{st} = \infty.
					    \end{array}
					    \right.
$$

Let  $O_B(X_S)$ to be the orthogonal group of $X_S$ associated to $B$. For each $s \in S$ we define $ \sigma_s : X_S \rightarrow X_S$ by $\sigma_s(x) = x - 2B(e_s,x)e_s$. The map $\sigma :W \hookrightarrow O_B(X_S)$ defined by $s \mapsto \sigma_s$ is called \emph{the geometrical representation} of $(W,S)$.  Let  $\sigma^*$ be the \emph{contragredient action} of $\sigma$.
 Let $\alpha \in \Phi$.  We denote 
$$
H_{\alpha} := \{ f \in X_S^*~| ~f(\alpha)  = 0 \}~\text{~and~}~R_{\alpha} := \{ f \in X_S^*~| ~f(\alpha)  > 0 \}.
$$

Let $C$ be the intersection of all $R_{\alpha}$ for $\alpha \in \Delta$ and let $D =\overline{C}$. The \emph{Tits cone} $ \mathcal{U}(W)$ of $W$ is defined by 
$$
~\mathcal{U}(W) := \bigcup\limits_{w \in W}wD.
$$ 

If $w \in W$, $wC$ is called a \emph{chamber} of $\mathcal{U}(W)$. The action of $W$ on $\{wC, w \in W \}$ is simply transitive and then the chambers of $\mathcal{U}(W)$  are in bijection with the elements of $W$. The chamber $C$ is called \emph{the fundamental chamber} of $W$ and it corresponds to the identity element of $W$.  The set $\{H_{\alpha}~|~\alpha \in \Phi\}$ is called the Coxeter arrangement of $(W,S)$.

The \emph{Shi arrangement} of $W$ is defined by $\shi(W) = \{ H_\alpha ~|~ \alpha\in \Sigma \}.$ The \emph{Shi regions} of $W$ are the connected components in the Tits cone of the Shi arrangement, that is  the connected components of 
$$
\mathcal{U}(W) ~\backslash \bigcup_{\alpha \in \Sigma} H_{\alpha}.
$$

 Let $\mathcal R$ be a Shi region.  We say that a  wall $H$ of $\mathcal R$ is a {\em descent-wall} of $\mathcal R$ if $H$ separates $\mathcal R$ from the fundamental chamber $C$.   The {\em set of descent-roots of a Shi region $\mathcal R$} is
$$
\Sigma D(\mathcal R):=\{\alpha\in \Sigma(\mathcal R)\mid H_{\alpha} \textrm{ is a descent-wall of } \mathcal R\}.
$$

The {\em right descent set  of $w\in W$} is the set:
$
D_R(w)=\{s\in S\mid \ell(ws)=\ell(w)-1\} = \{s\in S\mid w(\alpha_s)\in \Phi^-\}.
$
Let $s\in D_R(w)$, then there is a reduced word for $w$ ending with $s$, that is $w=us$ with $u\in W$ and $\ell(w)=\ell(u)+1$.  Then
$
N(u)=N(w)\setminus\{u(\alpha_s)\}. 
$
The {\em set of right descent roots of $w$}, denoted by $ND_R(w)$, is
$ND_R(w)= -w(\{\alpha_s\mid s\in D_R(w)\})$.
In other words, $\beta\in ND_R(w)$ if and only if  $H_{\beta}$  is a wall of $wC$ that separates $C$ from $wC$.

 \subsection{Shi arrangement in affine Weyl groups}\label{aff}
 
 Assume that $\Phi_0 \subset V$ is crystallographic where $V$ is a Euclidean space with inner product $\langle~,~\rangle$ and let $\Delta_0$ be a simple system.  Let $(W_0,S_0)$ be the corresponding Weyl group and $(W,S)$ the corresponding affine Weyl group.  The Coxeter arrangement of $W$ is given by $\{H_{\alpha,k}~|~\alpha \in \Phi_0^+,~k \in \mathbb{Z}\}$ where $H_{\alpha,k} = \{x \in V~|~\langle x,\alpha \rangle = k\}$.  The chambers of this arrangement are called alcoves and are in bijection with the elements of the group. The alcove corresponding to $w \in W$ is denoted by $A_w$.
 The fundamental chamber of $W_0$ is defined by 
 $$
 C_\circ := \{x \in V~|~\langle x,\alpha \rangle > 0, ~\forall \alpha \in \Delta_0\}.
 $$

In \cite{Shi87} Shi gives a parametrization of the alcoves $A_w$ of $(W,S)$ in terms of $\Phi_0^+$-tuples, denoted $(k(w,\alpha))_{\alpha \in \Phi_0^+}$, subject to certain conditions.  
In \cite{Shi88}, Shi uses the above parametrization in order to describe the Shi regions of $W$. Let $\overline{\mathscr S}$ be the set of $\Phi_0^+$-tuples over $\{-,0,+\}$; its elements are called {\em sign types}. For $w\in W$, the function $\zeta: W\to \overline{\mathscr S}$ is defined by
 $\zeta(w)=(X(w,\alpha))_{\alpha\in\Phi_0^+}$ where
\begin{equation}\label{eq:Xw}
 X(w,\alpha) = 
 \left\{
 \begin{array}{cc}
 + &\textrm{if } k(w,\alpha)>0\\
 0  &\textrm{if } k(w,\alpha)=0\\
 - &\textrm{if } k(w,\alpha)<0 .
 \end{array}
 \right.
\end{equation}

 The {\em sign type $X(\mathcal R):=(X(\mathcal R,\alpha))_{\alpha\in\Phi_0^+}$ of a Shi region $\mathcal R$} is  $X(\mathcal R)=\zeta(w)$ for some $w\in W$ such that  $A_w \subset \mathcal R$. A sign type $X=(X_\alpha)_{\alpha\in\Phi_0^+}$ is said to be {\em admissible} if $X$ is in the image of $\zeta$.   
  An easy exercise consists to show that $\text{Shi}(W) = \{H_{\alpha,k}~|~\alpha \in \Phi_0^+,~k =0,1\}$.  
  This arrangement was introduced in \cite{Shi86} in type $A$ and then in any type in \cite{Shi87}.  
 The connected components of this arrangement were then called admissible sign types but in our terminology those are the Shi regions of $W$.   A Shi region is called dominant if it is included in the fundamental chamber of $W_0$.  It is obvious that for a Shi region $\mathcal R$, $\mathcal R$ is dominant if and only if $X(\mathcal R,\alpha)\in \{+,0\}$ for all $\alpha\in\Phi_0^+$.

 Finally, Shi showed that any admissible sign type has a unique element of minimal length \cite{Shi88} and we call this element a minimal element.  We denote by $L_{\text{Shi}}(W)$, or just $L_{\text{Shi}}$, the set of minimal elements of $W$.  The minimal element of a dominant Shi region is called a dominant minimal element (see \cite{chaphohl2022shi} for detailed examples).

 The Weyl group $W_0$ is a standard parabolic subgroup of $W$, since $W_0$ is generated by $S_0\subseteq S$. The set of minimal length coset representatives  is ~${}^0W  := \{v\in W\mid \ell(sv)>\ell(v),  \ \forall s\in S_0\} $. The following well-known proposition states that an element is a minimal coset representative for $W_0\backslash W$ if and only if its corresponding alcove is in the dominant region. 

 \begin{prop}\label{prop:DomCR} Let $w\in W$. We have $w \in {}^0W$ if and only if $A_w \subseteq C_\circ$.
 \end{prop}

 \subsection{The two goals of this note}

 In \cite{chaphohl2022shi} we show with C. Hohlweg that the set of low elements of an affine Weyl group $W$ is equal to the set of minimal elements of $W$, solving a conjecture of Dyer and Hohlweg \cite[Conjecture 2]{DyHo16}.  Recently this conjecture has been solved for any Coxeter group \cite{dyer2023shi}. 
  To prove this conjecture in the affine setting we used the work of Cellini and Papi in their study of ad-nilpotent ideals of Borel subalgebras. Using their results, that we expressed in terms of root ideals rather than in terms of ad-nilpotent ideals, we first showed that any minimal element of a dominant Shi region is a low element. Then, we used an adequate parabolic decomposition to extend this result to other Shi regions.
  
  The goal of this note is to give a survey of the bijections involved in the study of dominant Shi regions and to give,  in terms of ad-nilpotent ideals, a detailed proof of the fact that dominant minimal elements are low elements. 
  
 We also ask the following question:
 
 \begin{question}
 Can we extend the notion of ad-nilpotent ideals so that the bijection between ad-nilpotent ideals and dominant Shi regions extends to the non-dominant Shi regions ?
 \end{question}

\medskip

\newpage

  \section{Survey on Shi regions, root ideals, antichains and ad-nilpotent ideals}\label{background ad-ideal}
  
  We are in the same setting as Section \ref{aff}.  Set $V=V_0\oplus \mathbb R \delta$ to be a real vector space with basis $\Delta_0\sqcup \{\delta\}$, where $\delta$ is an indeterminate.  We extend the inner product $\langle~,~\rangle$ into a symmetric bilinear form on $V$, that we also denote $\langle~,~\rangle$,  by $\langle x, \delta \rangle = 0$ for any $x \in V_0$ and $\langle \delta, \delta \rangle = 0$. 
The pair $(V,\langle~,~\rangle)$ is a quadratic space for which  the isotropic cone  is  $\mathbb R\delta= \{x\in V\mid \langle x,x\rangle=0\}$. 

A simple system in $V$ for $(W,S)$ is
$
\Delta = \Delta_0\sqcup \{\delta - \alpha_0\}.
$
The  root system $\Phi$ and the positive root system $\Phi^+$ for $(W,S)$ in $V$ are defined by $\Phi=\Phi^+\sqcup \Phi^-$, where $\Phi^-=-\Phi^+$ and
$\Phi^+ =( \Phi_0^+ +  \mathbb N\delta ) \sqcup ( \Phi_0^-+ \mathbb N^*\delta)$. 
 
 \subsection{background on Shi regions, root ideals and antichains} Recall that $\Psi\subseteq \Phi_0^+$ is a \emph{root ideal} of the poset  $(\Phi_0^+,\preceq)$ if for all $\alpha\in\Psi$ and $\gamma \in \Phi_0^+$ such that $\alpha\preceq \gamma$ we have $\gamma\in \Psi$.   By definition of the root poset, this is equivalent to: $\Psi\subseteq \Phi_0^+$ is a root ideal of $(\Phi_0^+,\preceq)$ if for all $\alpha\in\Psi$ and $\beta \in \Phi_0^+$ such that $\alpha+\beta\in \Phi_0^+$,  one has $\alpha+\beta\in \Psi$.   
 
 Denote by $\mathcal I(\Phi_0)$ the set of root ideals of $(\Phi_0^+,\preceq)$.  Let $Q$ be the root lattice of $\Phi_0$ and $Q^{\vee}$ be the coroot lattice.  In this article the affine Weyl group is $W=Q^{\vee} \rtimes W_0$ while in Shi's articles the affine Weyl group considered is $ W^{\vee} := Q \rtimes W_0$ even though Shi denotes it by $W$.
 
 \bigskip

$\bullet$ As explained in the introduction,  Shi shows in \cite{Shi88} that any Shi region $X$ has a unique element $m_X \in X$ of minimal length, characterized by a property of minimality of its Shi coefficients.  It follows the bijection:
$$
\begin{array}{ccc}
\{\text{Shi regions of}~ W\} & \longrightarrow & \{\text{minimal elements of}~W\}\\
                            X   				  & \longmapsto  & m_X.
\end{array}
$$
In particular this bijection restricts to the bijection:
$$
\begin{array}{ccc}
\{\text{dominant Shi regions of}~ W\} & \longrightarrow & \{\text{dominant minimal elements of}~W\}\\
                            X   				  & \longmapsto  & m_X.
\end{array}
$$
 
 \bigskip
 
$\bullet$ In \cite[Theorem 1.4]{Shi97}, Shi shows that there is a bijection between the set of root ideals of $(\Phi_0^+)^{\vee}$ (called \emph{increasing subsets}, and sometimes \emph{dual root ideals} in other references) and the set of dominant Shi regions of the affine Weyl group $W^{\vee}$ (called $\oplus$-sign types indexed by $\Phi_0^+$, or $\oplus$-sign types of $W^{\vee}$).  His proof his based on a characterization of $\oplus$-sign types  \cite[Proposition 1.2]{Shi97},  which is proven via a reduction to irreducible subroot systems of rank 2.  

Let $\Omega \in \{\Phi_0, \Phi_0^{\vee}\}$. For a root ideal $\Psi \in \mathcal{I}(\Omega)$ let us set $X_{\Psi} = (X(\Psi,\beta))_{\beta \in \Omega^+}$ where $\Omega^+ = \Phi_0^+$ if $\Omega = \Phi_0^{\vee}$,  $\Omega^+ = (\Phi_0^{\vee})^+$ if $\Omega = \Phi_0$ and 
$$
X(\Psi,\beta) = \left\{
\begin{array}{cc}
+&\textrm{if } \beta^{\vee} \in \Psi\\
0&\textrm{otherwise}.
\end{array}
\right. 
$$
Shi's bijection is then given by:
$$
\begin{array}{ccc}
\{\text{root ideals of}~ ((\Phi_0^{\vee})^+,\preceq)\} & \longmapsto & \{ \oplus\text{~-~sign types of }~W^{\vee}\}\\
                            \Psi       				  & \longmapsto  & X_{\Psi}.
\end{array}
$$
Then, by dualising this map we recover the bijection involving the affine Weyl group considered in this article, namely
$$
\begin{array}{ccc}
\{\text{root ideals of}~ (\Phi_0^+,\preceq)\} & \longrightarrow & \{ \oplus\text{~-~sign types of }~W\}\\
                            \Psi       				  & \longmapsto  & X_{\Psi}.
\end{array}
$$

Moreover, Shi provides a case by case enumeration of dominant Shi regions (i.e.,  $\oplus$-sign types).  To do so, he enumerates for each type $X$ of $\Phi_0^+$ all the possible root ideals of $(\Phi_0^+,\preceq)$ and expresses their corresponding dominant Shi regions.  Then,  after this enumeration he defines for any type $X$ the number $\mu(X)$ as the number of dominant Shi regions of type $X$ and gives explicit formulas of these numbers in types $A,B,C,D$: 

\begin{thm}[{\cite[Theorem 3.2]{Shi97}}]\label{number dom Shi regions}
$$
\mu(X) = \left\{
\begin{array}{cc}
\displaystyle\frac{1}{n+2}\binom{2n+2}{n+1}\text{~~}\text{~~}\textrm{if } X = A_n,\\
\displaystyle\binom{2n}{n} ~~\text{~~}\text{~~}\textrm{if } X = B_n, ~C_n,\\
 \displaystyle\binom{2n-1}{n}+\binom{2n-2}{n}  ~~\text{~~} \textrm{if } X = D_n.\\
\end{array}
\right. 
$$
\end{thm}

\medskip

Notice that $\mu(A_n)$ is the $(n+1)$-th Catalan number $C_{n+1}$.  He also gives in \cite[Theorem 3.6]{Shi97} the values of $\mu(X)$ when $X$ is exceptional,  but his method is not uniform and he ends his article by asking whether there exists a uniform formula for the numbers $\mu(X)$ \cite[Remark 3.7]{Shi97}.

\bigskip

$\bullet$ The notion of root ideal has also a strong connection with non-nesting and non-crossing partitions and the literature on the topic is extensive.  For example
it is well known that the number of non-crossing partitions and non-nesting partitions of $\{1,2,\dots,n\}$ is equal to the $n$-th Catalan number $C_n$.  

As explained in \cite[Remark 2]{Rei97},  in the non-nesting case Postnikov was able to redefine the non-nesting partitions of $\{1,2,\dots,n\}$ in terms of the root system of type $ A_{n-1}$ in a way that generalizes to all Weyl groups.  To do so,  Postnikov defines a non-nesting partition as an antichain of $(\Phi_0^+,\preceq)$,  that is a subset of $\Phi_0^+$ where the elements are mutually incomparable.

Postnikov also noticed the existence of a bijection between the set of antichains of $(\Phi_0^+,\preceq)$ and the set of dominant Catalan regions \cite{At98}, that is the regions of the Catalan hyperplane arrangement that lie in the fundamental chamber of $W_0$.  Let us denote by $\text{Cat}(\Phi_0)$ the number of antichains of $(\Phi_0^+,\preceq)$.  By the above bijection,  $\text{Cat}(\Phi_0)$ also counts the number of  dominant Catalan regions, which explains the terminology.

 However,  it turns out that the set of dominant Catalan regions is the same set as the set of dominant Shi regions,  therefore by Theorem \ref{number dom Shi regions} we know  $\text{Cat}(\Phi_0)$ for $\Phi_0 \in \{A_n,B_n, C_n, D_n\}$.  Using these observations and a former result  \cite[Theorem 5.5]{At96},  an intermediate answer of Shi's question \cite[Remark 3.7]{Shi97} was given by Athanasiadis:

\medskip

\begin{thm}[{\cite[Theorem 2.4]{At98}}]
Let $h$ be the Coxeter number of $\Phi_0$ and $e_1,\dots, e_n$ its exponents. The number of non-nesting partitions (that is antichains or dominant Shi regions) in type $X \in \{A_n, B_n, C_n, D_n\}$ is given by
$$
\prod\limits_{i=1}^n \displaystyle\frac{e_i+h+1}{e_i+1}.
$$
\end{thm}

\bigskip

Let us now express Postnikov's bijection.  Recall that $C_{\circ}$ is the fundamental chamber of $W_0$ and $W$ is the affine Weyl group associated to $W_0$.  For an antichain $A$ of $(\Phi_0^+,\preceq)$ we set:
$$
\mathcal{R}_A := \{x \in C_{\circ}~|~\langle \beta, x \rangle > 1~\text{if}~\alpha \preceq \beta~\text{for some}~\alpha \in A,~\langle \beta, x \rangle < 1~\text{otherwise}\}.
$$
 Postnikov's bijection is then given by:
$$
\begin{array}{ccc}
\{\text{antichains of}~ (\Phi_0^+,\preceq)\} & \longrightarrow & \{\text{dominants Shi regions of} ~W\}\\
                            A       				  & \longmapsto  & \mathcal{R}_A.
\end{array}
$$

\bigskip

 \subsection{Background on ad-nilpotent ideals}

  In \cite[Theorem 1]{CePa02}, Cellini and Papi give an explicit bijection between  the sets of ad-nilpotent ideals of a Borel subalgebra $\mathfrak{b}$ of simple Lie algebra $\mathfrak{g}$ and the $W_0$-orbits of $Q/(h+1)Q$ where $W_0$ is the Weyl group of $\mathfrak{g}$, $h$ the Coxeter number of $W_0$ and $Q$ the coroot lattice associtated to $\Phi_0$.  Recall that an ideal of $\mathcal{J}$ of $\mathfrak{b}$ is called an ad-nilpotent ideal if it is contained in the nilradical $[\mathfrak{b},\mathfrak{b}]$ of $\mathfrak{b}$.

Using this bijection and some results of Haiman from \cite[Section 7.4]{Ha94},  they also provide a formula that counts the number of ad-nilpotent ideals of $\mathfrak{b}$:

\begin{thm}[{\cite[Theorem 1, Formula (1)]{CePa02}}]
Let $h$ be the Coxeter number of $\Phi_0$ and $e_1,\dots, e_n$ its exponents. The number of ad-nilpotent ideals of $\mathfrak{b}$ is given by
$$
\displaystyle\frac{1}{|W_0|}\prod\limits_{i=1}^n (e_i+h+1).
$$
\end{thm}

Furthermore,  they also explain \cite[\S3 and \S4(1)]{CePa02} that ad-nilpotent ideals of $\mathfrak{b}$ are naturally in bijection with root ideals of $(\Phi_0^+,\preceq)$.  The bijection is as follows: first, they show that for any ad-nilpotent ideal $\mathcal{J}$ of $\mathfrak{b}$ there exists $\Phi_{\mathcal{J}} \subseteq \Phi_0^+$ such that $\mathcal{J} = \bigoplus_{\alpha \in \Phi_{\mathcal{J}}} \mathfrak{g}_{\alpha}$ where  $\mathfrak{g}_{\alpha}$ is the usual root space of $\mathfrak{g}$ associated to $\alpha$.  Then they show that the set $\Phi_{\mathcal{J}}$ is a root ideal of $(\Phi_0^+,\preceq)$.  Thus,  Cellini-Papi's bijection is  given by
$$
\begin{array}{ccc}
\{\text{ad-nilpotent ideals of}~ \mathfrak{b}\} & \longrightarrow & \{\text{root ideals of}~ (\Phi_0^+,\preceq)\}  \\
                             \mathcal{J}        				  & \longmapsto  & \Phi_{\mathcal{J}}.
\end{array}
$$

  As recalled in \cite[\S4 (1)]{CePa02}, it is a general fact from the theory of finite posets, that order ideals of a finite poset $(P, \leq)$ and antichains of $(P, \leq)$ are in bijection by mapping an order ideal $\Psi$ to the set of its minimal elements, which is defined by $\Psi_{\min}:=\{x \in \Psi~|~\text{if}~ y \leq x~\text{for some}~ y \in \Psi~\text{then}~x = y\}$.  Using this in the setting of root ideals we obtain the following bijection:
  $$
\begin{array}{ccc}
\{\text{root ideals of}~ (\Phi_0^+,\preceq)\} & \longrightarrow & \{\text{antichains of}~(\Phi_0^+,\preceq)\}\\
                             \Psi     				  & \longmapsto  & \Psi_{\min}.
\end{array}
$$

Finally,  \cite[Theorem 1]{CePa02} and its connections with root ideals, antichains and $\oplus$-sign types, provides a uniform enumeration of sign types in any type and gives a final answer to the question asked by Shi \cite[Remark 3.7]{Shi97}.

  \bigskip
  
  $\bullet$ In \cite{CePa00}, Cellini and Papi obtain a result that is of particular interest for us \cite[Theorem 2.6]{CePa00}.  First we need the following definitions, given in \cite[Section 2]{CePa00}: let $\Psi_1,  \Psi_2$ be two root ideals of $\Phi_0^+$.  We define a bracket relation on $\mathcal{I}(\Phi_0)$ by 
  $$
  [\Psi_1, \Psi_2] = \{\alpha_1+\alpha_2~|~\alpha_i \in \Psi_i~\text{and}~\alpha_1 + \alpha_2 \in \Phi_0^+\}
  $$
  and we set $\Psi^k = [\Psi^{k-1},\Psi]$.  This being defined, they give the following definition \cite[Definition 2.2]{CePa00} 
$$
  \mathbb L_\Psi:=\bigcup_{k\in \mathbb N^*} (k\delta-\Psi^k).
$$

We can now state their theorem, which is key in the following:

   \begin{thm}[{\cite[Theorem 2.6]{CePa00}}] \label{thm:CePa} Let $\mathcal{J}$ be an ad-nilpotent ideal of $\mathfrak{b}$ and let $\mathbb{L}_{\mathcal{J}} :=\mathbb{L}_{\Phi_{\mathcal{J}}}$. There is a unique element $w_{\mathcal{J}} \in W$ such that $$
\mathbb{L}_{\mathcal{J}} = N(w_{\mathcal{J}}).$$
\end{thm}

\medskip

The two following propositions sum up to Proposition \ref{prop CP}, which is essential in order to prove Proposition \ref{cor:LLshiDom} ($iii$).  We point out that Propositions \ref{PropCP1} and  \ref{PropCP2}  are given with a different terminology in their original version.
  
\medskip
  
  \begin{prop}[{\cite[Proposition 2.12(ii)]{CePa00}}]\label{PropCP1}
  Let $\mathcal{J}$ be an ad-nilpotent of $\mathfrak{b}$.  We have 
  $$
  ND_R(w_{\mathcal{J}}) \subseteq \{\delta - \alpha~|~\alpha \in \Phi_{0}^+\}.
  $$
  \end{prop}
  
  \medskip
  
    \begin{prop}[{\cite[Proposition 3.4]{CePa04}}]\label{PropCP2}
  Let $\mathcal{J}$ be an ad-nilpotent of $\mathfrak{b}$.  We have 
  $$
  ND_R(w_{\mathcal{J}}) \cap  \{\delta - \alpha~|~\alpha \in \Phi_0^+\} =  \{\delta - \alpha~|~\alpha \in (\Phi_{\mathcal{J}})_{\min}\}.
  $$
  \end{prop}
  
    \medskip

  \begin{prop}[Cellini-Papi]\label{prop CP}
    Let $\mathcal{J}$ be an ad-nilpotent of $\mathfrak{b}$.  We have 
  $$
  ND_R(w_{\mathcal{J}}) = \{\delta - \alpha~|~\alpha \in (\Phi_{\mathcal{J}})_{\min}\}.
  $$
  \end{prop}

  \bigskip
  
  \section{Dominant minimal elements are low}

This section is dedicated to show that the notion of minimal elements of dominant Shi regions is the same notion as that of low elements in dominant Shi regions (see Proposition \ref{equalities dom}).
  We denote by $L^0_\shi$ the set of minimal elements in dominant Shi regions of $W$,  $L^0$ the set of low elements in dominant Shi regions of $W$ and $\mathcal{I}_{\text{ad}}$ the set of ad-nilpotent ideals of $\mathfrak{b}$.

  \bigskip

\begin{prop} \label{prop dominant min element}
 Let $\mathcal{J} \in \mathcal{I}_{\text{ad}}$.  We have  $N(w_{\mathcal{J}}) =  \cone_{\Phi}(\{\delta-\alpha~|~\alpha \in \Phi_{\mathcal{J}}\})$.  In particular $w_{\mathcal{J}}$ is a low element of $W$.  
\end{prop}

  \begin{proof}
  By Theorem \ref{thm:CePa} we now there exists a unique element $w_{\mathcal{J}} \in W$ such that $\mathbb{L}_{\mathcal{J}}= N(w_{\mathcal{J}})$.  Thus, we need to show that $\mathbb{L}_{\mathcal{J}}= \cone_{\Phi}(\{\delta-\alpha~|~\alpha \in \Phi_{\mathcal{J}}\})$.  Let $x \in \mathbb{L}_{\mathcal{J}}$, then $x = k\delta - \beta$ with $k \in \mathbb{N}^*$ and $\beta \in \Phi_{\mathcal{J}}^k$.  Therefore, there exist $\beta_1,\dots,\beta_k \in \Phi_{\mathcal{J}}$ such that $\beta = \beta_1 + \dots \beta_k$.  It follows then that $x = (\delta-\beta_1) + \dots + (\delta-\beta_k)$, which means that $x \in \cone_{\Phi}(\{\delta-\alpha~|~\alpha \in \Phi_{\mathcal{J}}\})$. The reverse inclusion is almost the same and is left to the reader.
  
  But now,  we know that $w_{\mathcal{J}}$ is low if and only if there exists a subset $X \subset \Sigma$ such that $\text{cone}_{\Phi}(X) = N(w_{\mathcal{J}})$.  Since $\Sigma = \{\alpha, \delta-\alpha~|~\alpha \in \Phi_0^+\}$ and since $\Phi_{\mathcal{J}} \subseteq \Phi_0^+$,  by setting $Y := \{\delta-\alpha~|~\alpha \in \Phi_{\mathcal{J}}\}$ we do have $\text{cone}_{\Phi}(Y) = N(w_{\mathcal{J}})$ with $Y \subset \Sigma$,  implying then that $w_{\mathcal{J}}$ is low.
  \end{proof}
  
  \bigskip
  
  \begin{prop}\label{prop min papi dominant}
  Let $\mathcal{J}  \in \mathcal{I}_{\text{ad}}$ and let $\mathcal{R}_{\mathcal{J}}:=\mathcal{R}_{(\Phi_{\mathcal{J}})_{\min}}$.  The Shi region containing the alcove associated to $w_{\mathcal{J}}$ is dominant and is equal to $\mathcal{R}_{\mathcal{J}}$. Moreover,  if $\alpha \in \Phi_{\mathcal{J}}$ then $X(\mathcal R_{\mathcal{J}},\alpha) = +$ and if $\alpha \notin \Phi_{\mathcal{J}}$ then $X(\mathcal R_{\mathcal{J}},\alpha) = 0$.
  \end{prop}
  
  \begin{proof}
 By Proposition~\ref{prop:DomCR} we know that the alcove associated to $w_{\mathcal{J}}$ is in the fundamental chamber if and only if $N(w_{\mathcal{J}}) \cap \Phi_0^+ = \emptyset$.  But this equality is satisfied by Proposition \ref{prop dominant min element}.  Thus the Shi region containing $w_{\mathcal{J}}$ is dominant.

 Let us show now that the Shi region that contains the alcove associated to $w_{\mathcal{J}}$ is $\mathcal{R}_{\mathcal{J}}$.  To do so, we only need to show that for any $x \in w_{\mathcal{J}} \cdot A_\circ$ we have $\langle \beta, x \rangle > 1$ if $\alpha \preceq \beta$ for some $\alpha \in (\Phi_{\mathcal{J}})_{\min}$ and $\langle \beta, x \rangle > 1$ otherwise.  
 
 It turns out that we can rewrite the two conditions that define $\mathcal{R}_{\mathcal{J}}$ more easily.  Indeed,  if $\beta \notin \Phi_{\mathcal{J}}$ then in particular, since $\Phi_{\mathcal{J}}$ is a root ideal, $\beta$ cannot be greater or equal than any $\alpha \in \Phi_{\mathcal{J}}$ and let alone than any $\alpha \in (\Phi_{\mathcal{J}})_{\min}$.  Moreover,  if $\beta \in \Phi_{\mathcal{J}}$ it is clear that $\beta$ is either greater than a minimal element of $\Phi_{\mathcal{J}}$ or is a minimal element of $\Phi_{\mathcal{J}}$, in both cases we have for any $x \in \mathcal{R}_{\mathcal{J}}$ that $\langle \beta, x \rangle > 1$.  Hence
 \begin{align}\label{RJ}
  \mathcal{R}_{\mathcal{J}} = \{x \in C_{\circ}~|~\langle\ \beta, x \rangle > 1~\text{if}~\beta \in \Phi_{\mathcal{J}}~\text{and}~\langle \beta, x \rangle < 1~\text{if}~\beta \notin \Phi_{\mathcal{J}} \}.
 \end{align}

Let $x \in w_{\mathcal{J}} \cdot A_\circ$. Since $N(w_{\mathcal{J}}) =  \cone_{\Phi}(\{\delta-\alpha~|~\alpha \in \Phi_{\mathcal{J}}\})$ by Proposition \ref{prop dominant min element}, we have $k(w_{\mathcal{J}}, \beta) > 1$ for any $\beta \in \Phi_{\mathcal{J}}$ and $0\leq k(w_{\mathcal{J}}, \beta) < 1$ for any $\beta \notin \Phi_{\mathcal{J}}$.  In particular it follows that $\langle \beta, x \rangle > 1$ for any $\beta \in \Phi_{\mathcal{J}}$ and  $\langle \beta, x \rangle < 1$ for any $\beta \notin \Phi_{\mathcal{J}}$.  Thus $x \in \mathcal{R}_{\mathcal{J}}$ by Eq.  (\ref{RJ}). 

 With respect to the last statement, it is a direct consequence of the fact that $\mathcal{R}_{\mathcal{J}}$ is dominant and Eq.  (\ref{RJ}). This ends the proof.
  \end{proof}

  \bigskip

\begin{ex}
Let $\mathcal{J}$ be the ad-nilpotent ideal in type $A_3$ given by $$\mathcal{J} = \mathfrak{g}_{23} \oplus \mathfrak{g}_{13} \oplus \mathfrak{g}_{24} \oplus \mathfrak{g}_{14}.$$  The corresponding root ideal is $\Phi_{\mathcal{J}} = \{e_{23}, e_{13}, e_{24},e_{14}\}$, where $e_{ij} = e_i - e_j \in \mathbb{R}^4$,  and the associated antichain is $(\Phi_{\mathcal{J}})_{\min} = \{e_{23}\}$.  Thus, the corresponding dominant Shi region (given in terms of its $\oplus$-sign type) and its minimal element are

\medskip

\begin{figure}[h!]
\begin{center}
\begin{tikzpicture} 
\node at  (-1.7,0.5) {$ X(\mathcal{R}_{\mathcal{J}})$} ;
\node at (-0.7,0.5) {$=$} ;
\node at (0,0) {$0$} ;
\node at (1,0) {$+$} ;
\node at (2,0) {$0$} ;

\node at (0.5,0.5 ) {$+$} ;
\node at (1.5, 0.5 ) {$+$} ;

\node at (1,1) {$+$} ;

\node at  (4,0.5) {$w_{\mathcal{J}}$} ;
\node at (4.8,0.5) {$=$} ;
\node at (5.5,0) {$0$} ;
\node at (6.5,0) {$1$} ;
\node at (7.5,0) {$0$} ;

\node at (6,0.5 ) {$1$} ;
\node at (7, 0.5 ) {$1$} ;

\node at (6.5,1) {$1$} ;

\node at (8.3,0.5) {$=$} ;
\node at (9.5,0.5) {$s_2s_3s_1s_2$} ;
\end{tikzpicture}
\end{center}
\end{figure}
\end{ex}

  \bigskip

  \begin{prop}\label{equalities dom}
We have the following equalities
$$\{w_{\mathcal{J}}~|~\mathcal{J} \in \mathcal{I}_{\text{ad}}\} = L^0 = L^0_{\shi}.$$
\end{prop}

\begin{proof}
By \cite[Proposition 4.5]{chaphohl2022shi} we know that $L \subseteq L_{\shi}$.  It follows that $L^0 \subseteq L_{\shi}^0$.  By Proposition \ref{prop dominant min element} we have $\{w_{\mathcal{J}}~|~\mathcal{J} \in \mathcal{I}_{\text{ad}}\} \subseteq L$.  Moreover, by Proposition \ref{prop min papi dominant} we also know that for any ad-nilpotent ideal $\mathcal{J}$, the element $w_{\mathcal{J}}$ is in a dominant Shi region.  This implies that $\{w_{\mathcal{J}}~|~\mathcal{J} \in \mathcal{I}_{\text{ad}}\} \subseteq L^0$. 
On the other hand,  we know by using all the previous bijections that $|L^0_{\shi}| = | \{w_{\mathcal{J}}~|~\mathcal{J} \in \mathcal{I}_{\text{ad}}\}|$ (see Figure \ref{synthesis bij}).  The result follows. 
\end{proof}

  \bigskip

\begin{prop}\label{cor:LLshiDom}
 Let $\mathcal{J}  \in \mathcal{I}_{\text{ad}}$. 
 \begin{itemize}
 \item[$(i)$] The set $(\Phi_{\mathcal{J}})_{\min}$ is the set of roots $\alpha$  such that the sign type $X(\mathcal R_{\mathcal{J}})$ can be changed into a new admissible sign type $X(\mathcal R')$ (for another Shi region $\mathcal{R}'$) defined by $X(\mathcal R', \beta) = X(\mathcal R_{\mathcal{J}},\beta)$ if $\beta \neq \alpha$ and $X(\mathcal R', \alpha)=0$.
 \item[$(ii)$]   $\Sigma D(\mathcal R_{\mathcal{J}}) = \{\delta - \alpha~|~\alpha \in (\Phi_{\mathcal{J}})_{\min}\}$.
 \item[$(iii)$]  $\Sigma D(\mathcal R_{\mathcal{J}})= ND_R(w_{\mathcal{J}})$. 
 \end{itemize}
 \end{prop}

\begin{proof}
$(i)$ First of all, we know by Proposition \ref{prop min papi dominant} that for any $\alpha \in \Phi_{\mathcal{J}}$ we have $X(\mathcal{R}_{\mathcal{J}},  \alpha) = +$.  Let $\alpha \in (\Phi_{\mathcal{J}})_{\min}$ such that we cannot change $X(\mathcal R_{\mathcal{J}},\alpha) = +$ into $X(\mathcal R',\alpha) =0$.  This means then that there exists an irreducible subroot system of rank 2 of $\Phi_0^+$, say $\Psi$, such that $\alpha \in \Psi$ and $\alpha = \beta_1 + \beta_2$ with $\beta_1, \beta_2 \in  \Psi$ and at least one $\beta_i$, say $\beta_1$,  such that  $X(\mathcal R_{\mathcal{J}},\beta_1) = +$.  But then it follows that $\beta_1 \in \Phi_J$, contradicting the minimality of $\alpha$.   Hence the first inclusion. The reverse inclusion is similar. 

$(ii)$ First of all, since $\mathcal R_{\mathcal{J}}$ is a dominant Shi region by Proposition \ref{prop min papi dominant},  it follows that $\Sigma D(\mathcal R_{\mathcal{J}}) \subset \{\delta - \alpha~|~\alpha \in \Phi_0^+\}$.  Now, by Definition $\Sigma D(\mathcal R_{\mathcal{J}})$ is the set of roots associated to the descent-walls of $\mathcal R_{\mathcal{J}}$.  In other words,  $\Sigma D(\mathcal R_{\mathcal{J}})$ is the set of roots $\delta - \alpha$ such that the corresponding hyperplanes $H_{\alpha,1}$ are adjacent to the Shi region $\mathcal R_{\mathcal{J}}$ and to another one, say $\mathcal{R}'$.  But then it follows that the only difference between the sign types $X(\mathcal R_{\mathcal{J}})$ and  $X(\mathcal R')$ is over the position $\alpha$ with   $X(\mathcal R_{\mathcal{J}},\alpha) = +$ and $X(\mathcal R',\alpha) = 0$.  The equality follows then by the first point $(i)$.

$(iii)$ By Proposition \ref{prop CP} we know that $ND_R(w_{\mathcal{J}}) = \{\delta - \alpha~|~\alpha \in (\Phi_{\mathcal{J}})_{\min}\}$. The result follows then by the second point $(ii)$.  
\end{proof}
 
\medskip
 
 \begin{ex}
 Let $\mathcal{J}$ be the ad-nilpotent ideal with corresponding sign type given in Figure \ref{example min antichain}.  The set $\Phi_{\mathcal{J}}$ is the of roots $\alpha$ where $X(\mathcal{R}_{\mathcal{J}},\alpha) = +$, and the corresponding antichain is given by $(\Phi_{\mathcal{J}})_{\text{min}} = \{e_{23}, e_{35}, e_{46}\}$ (the red positions below).  We see that if change a red $+$ into a $0$ then the new sign type is again admissible, whereas if we change a black $+$ into a $0$ then we lose the admissibility.
  \end{ex}  
 \begin{figure}[h!]
\begin{center}
\begin{tikzpicture} 
\node at (0,0) {$0$} ;
\node at (1,0) {$\textcolor{red}{+}$} ;
\node at (2,0) {$0$} ;
\node at (3,0) {$0$} ;
\node at (4,0) {$0$} ;
\node at (5,0) {$0$} ;

\node at (0.5,0.5 ) {$+$} ;
\node at (1.5, 0.5 ) {$+$} ;
\node at (2.5, 0.5) {$\textcolor{red}{+}$} ;
\node at (3.5, 0.5) {$\textcolor{red}{+}$} ;
\node at (4.5, 0.5) {$0$} ;

\node at (1,1) {$+$} ;
\node at (2,1) {$+$} ;
\node at (3,1) {$+$} ;
\node at (4,1) {$+$} ;

\node at (1.5,1.5 ) {$+$} ;
\node at (2.5,1.5 ) {$+$} ;
\node at (3.5,1.5 ) {$+$} ;

\node at (2,2) {$+$};
\node at (3,2) {$+$};

\node at (2.5,2.5) {$+$};
\end{tikzpicture}
\end{center}
\caption{The $\oplus$-sign type of a dominant Shi region in the affine Weyl group of type $A_6$.  The reader may read \cite{chaphohl2022shi, Shi86, Shi87, Shi88} for a more detailed explanation of the above figure.}
\label{example min antichain}
\end{figure}
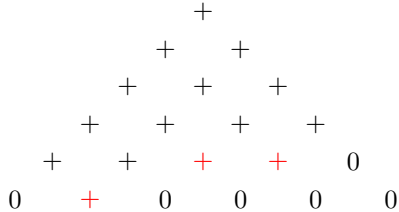

\begin{center}
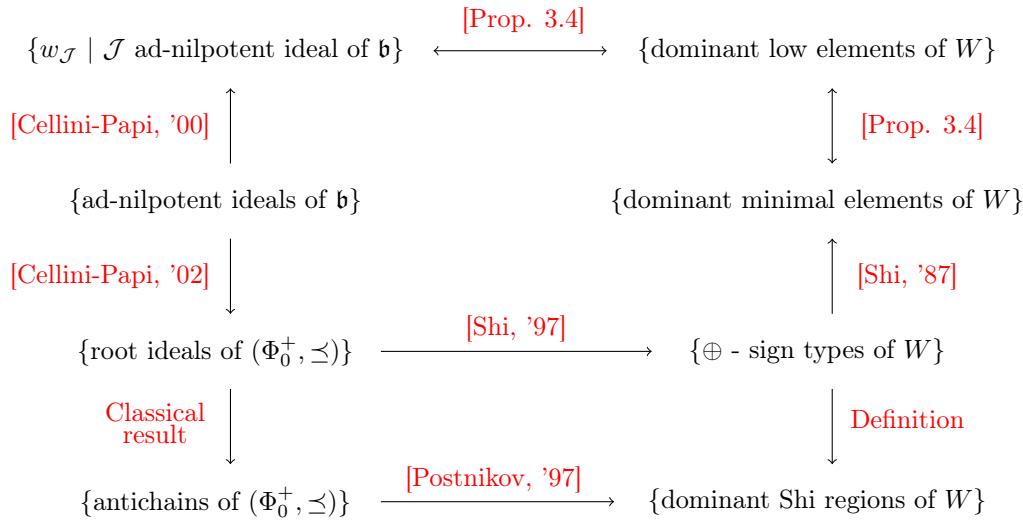
\begin{figure}[h!]
\begin{tikzpicture} 
\node at (0, 0) {$\{\text{root ideals of}~ (\Phi_0^+,\preceq)\}$} ;
\node at (8, 0) {$\{ \oplus\text{~-~sign types of}~W\}$} ; 
\node at (0, -2) {$\{\text{antichains of}~ (\Phi_0^+,\preceq)\}$} ;
\node at (8, -2) {$\{ \text{dominant Shi regions of}~W\}$} ; 
\node at (0, 2) {\{$\text{ad-nilpotent ideals of}~ \mathfrak{b}\}$} ;
\node at (8, 2) {\{$\text{dominant minimal elements of}~W\}$} ;
\node at (8, 4) {\{$\text{dominant low elements of}~W\}$} ;
\node at (0, 4) {$\{ w_{\mathcal{J}}~|~\mathcal{J}~\text{ad-nilpotent ideal of}~ \mathfrak{b}\}$} ;

\node at (4, 0.3) {\textcolor{red}{[Shi,  '97]}} ;
\node at (3.7, -1.7) {\textcolor{red}{[Postnikov,  '97]}} ;
\node at (9.2, -0.9) {\textcolor{red}{Definition}} ;
\node at (-0.8, -0.8) {\textcolor{red}{Classical}} ;
\node at (-0.8, -1.1) {\textcolor{red}{result}} ;
\node at (-1.4, 1) {\textcolor{red}{[Cellini-Papi, '02]}} ;
\node at (9.2, 1) {\textcolor{red}{[Shi,  '87]}} ;
\node at (-1.4, 3) {\textcolor{red}{[Cellini-Papi, '00]}} ;
\node at (9.4, 3) {\textcolor{red}{[Prop.  \ref{equalities dom}]}} ;
\node at (4.1, 4.4) {\textcolor{red}{[Prop.  \ref{equalities dom}]}} ;

\draw [->] (0.2,-0.5) -- (0.2,-1.5) ;
\draw [->] (8.2,-0.5) -- (8.2,-1.5) ;
\draw [->] (0.2,1.5) -- (0.2,0.5) ;
\draw [<-] (8.2,1.5) -- (8.2,0.5) ;
\draw [->] (0.2,2.5) -- (0.2,3.5) ;
\draw [<->] (8.2,2.5) -- (8.2,3.5) ;

\draw [->] (2.2,0) -- (5.8,0) ;
\draw [->] (2.2,-2) -- (5.3,-2) ;
\draw [<->] (2.9,4) -- (5.2,4) ;

\end{tikzpicture}
\caption{The bijections involved in the study of the dominant Shi regions.}
\label{synthesis bij}
\end{figure}
\end{center}

\bigskip
\vspace{2cm}
\bigskip

\bigskip
\vspace{2cm}
\bigskip

\bibliographystyle{plain}
\bibliography{note.bib}

\end{document}